\documentclass[11pt,a4paper]{article}
\usepackage[utf8]{inputenc}
\usepackage[all]{xy}
\usepackage{latexsym}
\usepackage[margin=3cm]{geometry}
\usepackage[dvipdfmx]{graphicx}
\usepackage{authblk}
\usepackage{hyperref}
\usepackage{mymacro}
\usepackage{tikz}
\usepackage{mathtools}
\usepackage[
    backend=biber,
    style=ext-alphabetic,
    maxbibnames=5,
    sorting=nyvt,
    giveninits,
    articlein=false,
    url=true,  
    doi=false,
    eprint=true,
    isbn=false,
]{biblatex}
\SelectTips{cm}{}

\title{The $\gamma$-support as a micro-support}
\author[1]{Tomohiro Asano\thanks{\texttt{tasano@se.kanazawa-u.ac.jp}, \texttt{tomoh.asano@gmail.com}}}
\author[2]{St\'ephane Guillermou\thanks{\texttt{Stephane.Guillermou@univ-nantes.fr}. Also supported by ANR COSY (ANR-21-CE40-0002)}}
\author[3]{Vincent Humili\`ere\thanks{\texttt{vincent.humiliere@imj-prg.fr}. Also supported by ANR COSY (ANR-21-CE40-0002) and ANR CoSyDy (ANR-CE40-0014)}} 
\author[4]{ \\ Yuichi Ike\thanks{\texttt{ike@mist.i.u-tokyo.ac.jp}, \texttt{yuichi.ike.1990@gmail.com}. Also supported by JSPS KAKENHI (21K13801 and 22H05107)}}
\author[5]{Claude Viterbo\thanks{\texttt{Claude.Viterbo@universite-paris-saclay.fr}. Also supported by ANR COSY (ANR-21-CE40-0002)}}
\affil[1]{Faculty of Electrical, Information and Communication Engineering,
Institute of Science and Engineering, Kanazawa University, Kakumamachi, Kanazawa,
920-1192, Japan}
\affil[2]{UMR CNRS 6629 du CNRS Laboratoire de Mathématiques Jean LERAY
2 Chemin de la Houssinière, BP 92208, F-44322 NANTES Cedex 3 France}
\affil[3]{Sorbonne Universit\'e and Université de Paris, CNRS, IMJ-PRG, F-75005 Paris, France
 and Institut Universitaire de France.}
\affil[4]{Graduate School of Information Science and Technology, The University of
Tokyo, 7-3-1 Hongo, Bunkyo-ku, Tokyo 113-8656, Japan}
\affil[5]{ Universit\'e Paris-Saclay, CNRS, Laboratoire de mathématiques d'Orsay, 91405, Orsay, France. }
\date{\today}

\begin{document}

\maketitle

\begin{abstract}
    We prove that for any element $L$ in the completion of the space of smooth compact exact Lagrangian submanifolds of a cotangent bundle equipped with the spectral distance, the $\gamma$-support of $L$ coincides with the reduced micro-support of its sheaf quantization.
    As an application, we give a characterization of the Vichery subdifferential in terms of $\gamma$-support. 
\end{abstract}

\section{Introduction}

Let $M$ be a $C^\infty$ closed manifold. The space $\frakL(T^*M)$ of smooth
compact exact Lagrangian submanifolds of $T^*M$ carries a distance, $\gamma$, called the spectral distance. 
This was introduced by Viterbo \cite{viterbo1992} for the class of Lagrangians that are Hamiltonian isotopic to the zero section, and later extended\footnote{In the Floer context, the extension follows from the spectral invariants defined in \cite{Oh}, \cite{M-V-Z}, and \cite{H-L-S} by using the main result of \cite{F-S-S}. One can also use later results in the sheaf framework, by using the spectral invariants defined by Vichery in \cite{Vichery-2} and the quantization of exact Lagrangians from \cite{Guillermou} (see also \cite{Viterbo-Sheaves}).} to $\frakL(T^*M)$. The metric space $(\frakL(T^*M), \gamma)$ is not complete (not even a Baires space \cite[App.~A]{viterbo2022supports}), and we are interested in this note in its completion. 
Its study was initiated in \cite{Humiliere-completion}, pursued further in \cite{viterbo2022supports}, and has applications to Hamilton-Jacobi equations \cite{Humiliere-completion}, Symplectic Homogenization theory \cite{Viterbo-Homogenization}, and to conformally symplectic dynamics \cite{Arnaud-Humiliere-Viterbo}.

The elements of the completion $\widehat\frakL(T^*M)$ are by definition certain equivalence classes of
Cauchy sequences with respect to the spectral norm $\gamma$. Despite their very abstract nature,
they admit a geometric incarnation first introduced by Humili\`ere in \cite{Humiliere-completion} and in a
different version called \emph{$\gamma$-support} much more recently by Viterbo in \cite{viterbo2022supports}. For a smooth Lagrangian $L\in\frakL(T^*M)$ we have $\gammasupp(L)=L$. We refer the reader to \cref{subsec:gamma-support} for the precise definition of the $\gamma$-support.

To each element of $\frakL(T^*M)$, it is also possible to associate an object $F_L$ in the derived category of sheaves $\SD (k_{M\times \mathbb R})$, which more precisely belongs to the so-called Tamarkin category. 
This was proved by Guillermou-Kashiwara-Schapira \cite{G-K-S} in the class of Lagrangians that are Hamiltonian isotopic to the zero section and later extended to $\frakL(T^*M)$ by Guillermou \cite{Guillermou} and Viterbo \cite{Viterbo-Sheaves}. The object $F_L$ is called \emph{sheaf quantization} of $L$. Conversely, to an object $F$ in the Tamarkin category, one can associate a closed subset of $T^*M$ which we call \emph{reduced micro-support} and denote $\RS(F)$, in such a way\footnote{In fact, the object $F_L$ is only defined up to shift, but $\RS(F_L)$ is well-defined. See \cref{subsec:SQ}.} that $\RS(F_L)=L$. 
The sheaf quantization can often be used instead of a generating function, and it is known to exist in more general cases (in particular for any exact embedded Lagrangian).
This approach allows, on one hand, to get rid of the ``Hamiltonian isotopic to the zero section'' condition required to use generating functions quadratic at infinity, and allows to prove or reprove a number of results in symplectic topology (see, for example, \cite{guillermou19}).

The correspondence $L \mapsto F_L$ was recently extended in \cite{GV2022singular} to the completion of $\frakL(T^*M)$ (see also \cite{AI22} for a similar result in different settings).
We therefore obtain two notions of support for an element $L$ in
$\widehat\frakL(T^*M)$, namely the $\gamma$-support of $L$ and the reduced
micro-support $F_L$. These two notions coincide on  $\frakL(T^*M)$, and it was asked by Guillermou and Viterbo ({\cite[Problem~9.10]{GV2022singular}}) whether they coincide in general. Our main result below answers positively this question. 

\begin{theorem}\label{th:main}
	For any $L \in \widehat{\frakL}(T^*M)$, one has 
	\begin{equation}
		\gammasupp(L) = \RS(F_L).
	\end{equation}
\end{theorem}

This allows us to understand the $\gamma$-support in terms of the more classical micro-support from sheaf theory. However, in some instances, we can also get information on the micro-support of a sheaf using $\gamma$-support (see Remark~\ref{Rmks:3.3}(\ref{Rmk:3.3-3})).

The main theorem is proved in \cref{sec:proof}. 
In \cref{sec:subdifferential}, we provide an application of this result to a characterization of the Vichery subdifferential defined in \cite{Vichery}.

\subsection*{Acknowledgments}

This work was carried out while Yuichi Ike was visiting IMJ-PRG and Jean Leray Mathematical Institute. 
He thanks both institutions for their hospitality during the visit.

\section{Preliminaries}

Let $M$ be a $C^\infty$ manifold and $\pi \colon T^*M \to M$ its cotangent bundle. 
We write $(x;\xi)$ for local coordinates of $T^*M$, the  Liouville form $\lambda$ is then defined by $\lambda=\sum_i \xi_i dx_i$. 
We denote the zero section of $T^*M$ by $0_M$.

\subsection{The \texorpdfstring{$\gamma$}{gamma}-support of elements in \texorpdfstring{$\widehat{\frakL}(T^*M)$}{L(TM)}}\label{subsec:gamma-support}

Let $\cL(T^*M)$ denote the set of compact exact Lagrangian branes, i.e., triples $(L, f_L, \widetilde G)$, where $L$ is a compact exact Lagrangian submanifold of $T^*M$, $f_L \colon L \to \bR$ is a function satisfying $df_L=\lambda_{\mid L}$, and $\widetilde G$ is a grading of $L$ (see \cite{Seidel-graded, viterbo2022supports}).
By abuse of notation, we simply write $L$ for an element $(L, f_L, \widetilde G)$ of $\cL(T^*M)$.
The action of $\bR$ on $\cL(T^*M)$ given by $(L,f_L, \widetilde G) \mapsto (L,f_L-c, \widetilde G)$ is denoted by $T_c$. 
For $L_1, L_2$ in $\cL(T^*M)$ we define as in \cite{viterbo2022supports} the spectral invariants $c_+(L_1,L_2)$ and $c_-(L_1,L_2)$ and finally 
\[
    \cdis (L_1,L_2)= \vert c_+(L_1,L_2)\vert + \vert c_-(L_1,L_2)\vert.
\]
Set $\mathfrak L(T^*M)$ to be the set of compact exact Lagrangians, where we do not record the primitive or grading.
For $L_1, L_2$ in $\mathfrak L(T^*M)$, we define 
\[
    \gamma(L_1,L_2)= \inf_{c\in \mathbb R} \cdis (L_1,T_c L_2)=c_+(L_1,L_2)-c_-(L_1,L_2).
\]
Denote by $\widehat{\frakL}(T^*M)$ (resp.\ $\widehat{\cL}(T^*M)$) the completion of $\frakL(T^*M)$ (resp.\ $\cL(T^*M)$) with respect to $\gamma$ (resp.\ $\cdis$).
We use the same symbol $T_c$ to mean the action on $\widehat{\cL}(T^*M)$ extending that on $\cL(T^*M)$.

Note that the standard action of the group of compactly supported Hamiltonian diffeomorphisms $\mathrm{Ham}_c(T^*M)$ on $\frakL(T^*M)$ given by $(\phi, L)\mapsto \phi(L)$ naturally extends to an action of $\mathrm{Ham}_c(T^*M)$ on the completion  $\widehat{\frakL}(T^*M)$. 
We are now ready to define the $\gamma$-support.

\begin{definition}[Viterbo \cite{viterbo2022supports}]\label{def:gamma-supp} 
    Let $L\in\widehat{\frakL}(T^*M)$. 
    The $\gamma$-support of $L$, denoted $\gammasupp(L)$, is the complement of the set of all $x\in T^*M$ which admit an open neighborhood $U$ such that $\phi(L)=L$ for any Hamiltonian diffeomorphism $\phi$ supported in $U$.
\end{definition}

When $L$ is a genuine smooth Lagrangian submanifold, i.e., belongs to $\frakL(T^*M)$, then $\gammasupp(L)=L$ (see \cite[Prop.~6.17.(1)]{viterbo2022supports}). In general, $\gammasupp(L)$ is a closed subset of $T^*M$ which can be very singular. However, not every closed subset can arise as a $\gamma$-support since $\gamma$-supports are always coisotropic in a generalized sense (called $\gamma$-coisotropic, see \cite[Thm.~7.12]{viterbo2022supports}). We will use the following property (see \cite[Prop.~6.20.(4)]{viterbo2022supports}): given smooth closed manifolds $M_1, M_2$, we have
\begin{equation}\label{eq:gammasupp-product}
    \gammasupp(L_1\times L_2)\subset\gammasupp(L_1)\times\gammasupp(L_2)
\end{equation}
for any $L_1\in\widehat{\frakL}(T^*M_1)$ and $L_2\in\widehat{\frakL}(T^*M_2)$.

We refer the interested reader to \cite{viterbo2022supports} for many further properties of the $\gamma$-support.

\subsection{Sheaf quantization of elements in \texorpdfstring{$\widehat{\cL}(T^*M)$}{L(TM)}}\label{subsec:SQ}

We fix a field $\bfk$ throughout the paper.
Given a $C^\infty$-manifold without boundary $X$,
we let $\SD(\bfk_X)$ denote the unbounded derived category of sheaves of $\bfk$-vector spaces on $X$. 
We denote by $\bfk_X$ the constant sheaf on $X$ with stalk $\bfk$.
For an inclusion $i \colon Z \hookrightarrow X$ of a locally closed subset, we also write $\bfk_Z$ for the zero-extension to $X$ of the constant sheaf on $Z$ with stalk $\bfk$. 
For an object $F \in \SD(\bfk_X)$, we denote by $\MS(F) \subset T^*X$ its micro-support, which is defined in \cite{KS90} (see also Robalo--Schapira~\cite{RS18} for the unbounded setting). 

We now recall the definition of the Tamarkin category~\cite{Tamarkin} (see also \cite{GS14}). 
We denote by $(t;\tau)$ the canonical coordinate on $T^*\bR_t$.
The Tamarkin category $\cD(M)$ is defined as the quotient category 
\[
	\SD(\bfk_{M \times \bR_t})/\SD_{\{\tau \le 0\}}(\bfk_{M \times \bR_t}),
\]
where $\SD_{\{\tau \le 0\}}(\bfk_{M \times \bR_t}) \coloneqq \{F \in \SD(\bfk_X) \mid \MS(F) \subset \{ \tau \leq 0 \}\}$ is the full triangulated subcategory of $\SD(\bfk_X)$.
The category $\cD(M)$ is equivalent to the left orthogonal ${}^\perp \SD_{\{\tau \le 0\}}(\bfk_{M \times \bR_t})$. 
For an object $F \in \cD(M)$, we define its reduced micro-support $\RS(F) \subset T^*M$ by 
\[
	\RS(F) \coloneqq \overline{\rho_t(\MS(F) \cap \{ \tau >0\})},
\]
where $\{ \tau >0\} \subset T^*(M \times \bR_t)$, $\rho_t \colon \{ \tau >0\} \to T^*M, (x,t;\xi,\tau) \mapsto (x;\xi/\tau)$, and $\overline{A}$ denotes the closure of $A$.

We can also describe the action of $\Ham_c(T^*M)$ on $\cD(M)$ as follows. 
Let $H \colon T^*M \times I \to \bR$ be a compactly supported Hamiltonian function and denote by $\phi^H=(\phi^H_s)_{s \in I} \colon T^*M \times I \to T^*M$ the Hamiltonian isotopy generated by $H$.
Then we can construct an object $K^H \in \SD(\bfk_{(M\times\bR)^2 \times I})$ whose micro-support coincides with the Lagrangian lift of the graph of $\phi^H$ outside the zero section (see \cite{G-K-S} for the definition).
For $s \in I$, we set $K^H_s \coloneqq K^H|_{(M \times \bR)^2 \times \{s\}} \in \SD(\bfk_{(M \times \bR)^2})$. 
We define a functor $\Phi^H_s \colon \cD(M) \to \cD(M) \ (s \in I)$ to be the composition with $K^H_s$.
For any $F \in \cD(M)$, we find that
\[
    \RS(\Phi^H_s(F))=\phi^H_s(\RS(F)).
\]

We now explain the sheaf quantization of an element of $\cL(T^*M)$.
For $L\in \cL(T^*M)$, we define 
\[
	\widetilde{L} \coloneqq \{ (x,t;\xi,\tau) \mid \tau >0, (x;\xi/\tau) \in L, t=-f_L(x;\xi/\tau) \}.
\]
Guillermou~\cite{Guillermou} (see also \cite{guillermou19, Viterbo-Sheaves}) proved the existence and the uniqueness of an object $F_L \in \cD(M)$ that satisfies $\MS(F_L) \setminus 0_{M \times \bR_t} =\widetilde{L}$ and $F_L|_{M \times (c,\infty)} \simeq \bfk_{M \times (c,\infty)}$ for a sufficiently large $c>0$. 
The object $F_L$ is called the sheaf quantization of $L$. 
We write this correspondence by $Q \colon \cL(T^*M) \to \cD(M), L \mapsto F_L$. 
Note that the grading of $L$ specifies the grading of $F_L$; in other words, $Q$ sends $L[k]$ to $F_L[k]$. 
However, we shall mostly forget about gradings here. 

We note that $\phi^H$ has a canonical lift to a homogeneous Hamiltonian isotopy of $T^*(M\times\bR) \setminus 0_{M\times\bR}$ (use the extra variable $\tau$ to make $H$ homogeneous), or equivalently, a contact isotopy of $J^1(M)$.  
In this way $\phi^H$ also acts on $\cL(T^*M)$. 
Moreover $\phi^H$ commutes with $T_c$ and it extends to $\widehat{\cL}(T^*M)$.
By the uniqueness, we find that
\begin{equation}\label{eq:QandPhi}
    Q(\phi^H_1(L)) \simeq \Phi^H_1(Q(L))     
\end{equation}
for any compact supported Hamiltonian function $H$.

We can define an interleaving-like distance $d_{\cD(M)}$ on the Tamarkin category $\cD(M)$ (see Kashiwara-Schapira~\cite{K-S-distance}, Asano--Ike~\cite{AI20,AI22}, and Guillermou--Viterbo~\cite{GV2022singular}). 
Since there are several different definitions and conventions for distance on $\cD(M)$, we give the definition here. 
For $c \in \bR$, we let $T_c \colon M \times \bR_t \to M \times \bR_t, (x,t) \mapsto (x,t+c)$ be the translation map to the $\bR_t$-direction by $c$. 
For an object $F \in \cD(M)$, we simply write $T_cF$ for ${T_c}_*F$.
Note that $Q$ sends $T_c L$ to $T_c F_L$.
For $F \in \cD(M)$ and $c \ge 0$, we have a canonical morphism $\tau_c(F) \colon F \to T_cF$ in $\cD(M)$ (see \cite{Tamarkin,GS14} for details).
Using the canonical morphisms, we can define the (pseudo-)distance on $\cD(M)$ as follows.

\begin{definition}
    Let $F,G \in \cD(M)$ and $a,b \ge 0$.
    \begin{enumerate}
        \item The pair $(F,G)$ is said to be \emph{$(a,b)$-isomorphic} if there exist morphisms $\alpha \colon F \to T_a G$ and $\beta \colon G \to T_b F$ in $\cD(M)$ such that 
        \[
        \begin{cases}
            \ld F \xrightarrow{\alpha} T_a G \xrightarrow{T_a \beta} T_{a+b} F \rd = \tau_{a+b}(F), \\
            \ld G \xrightarrow{\beta} T_b F \xrightarrow{T_b \alpha} T_{a+b} G \rd = \tau_{a+b}(G).
        \end{cases}
        \]
        \item We define 
        \[
            d_{\cD(M)}(F,G) \coloneqq \inf \lc a+b \relmid \text{$(F,G)$ is $(a,b)$-isomorphic} \rc.
        \]
    \end{enumerate}
\end{definition}

In Asano--Ike~\cite{AI22} and Guillermou--Viterbo~\cite{GV2022singular}, it is shown that $d_{\cD(M)}$ is complete\footnote{Since $d_{\cD(M)}$ is a pseudo-distance, the limit is not necessarily unique.}.
In Guillermou--Viterbo~\cite{GV2022singular}, Remark 6.12, it is also proved that for $L_1,L_2 \in \cL(T^*M)$
\begin{equation}\label{eq:ineq-c}
	d_{\cD(M)}(F_{L_1},F_{L_2}) \le \cdis (L_1,L_2) \le 2 d_{\cD(M)}(F_{L_1},F_{L_2}).
\end{equation}
Hence, using the completeness and the non-degeneracy of the distance for limits of constructible sheaves (\cite[Prop.~B.7]{GV2022singular}, we can extend $Q \colon \cL(T^*M) \to \cD(M)$ as 
\begin{equation}\label{eq:defQhat}
	\widehat{Q} \colon \widehat{\cL}(T^*M) \to \cD(M).
\end{equation}
We still write $F_L=\widehat{Q}(L)$ for $L \in \widehat{\cL}(T^*M)$.
Note that $\widehat Q$ also satisfies $\widehat Q(T_c L) \simeq T_c \widehat Q (L)$ for $L \in \widehat \cL(T^*M)$.
By a result of Viterbo~\cite[Prop.~5.5]{viterbo2022supports}, the canonical map $\widehat{\cL}(T^*M) \to \widehat{\frakL}(T^*M)$ is surjective, and two elements $L_1,L_2 \in \widehat{\cL}(T^*M)$ have the same image if and only if they coincide up to shift.
Hence, for $L \in \widehat{\frakL}(T^*M)$, the object $F_L \in \cD(M)$ is well-defined up to shift.
In particular, $\RS(F_L)$ is well-defined for $L \in \widehat{\frakL}(T^*M)$. 

Since the action of
  $\Phi^H$ on $\cL(T^*M)$ (or also $\cD(M)$) commutes with $T_c$, it  is an isometry. It follows that the extension of this action to $\widehat{\cL}(T^*M)$ still satisfies~\eqref{eq:QandPhi}:
\begin{equation}
\label{eq:QandPhi2}
    \widehat{Q}(\phi^H_1(L)) \simeq \Phi^H_1(\widehat{Q}(L)). 
\end{equation}

\section{Proof of the main result}\label{sec:proof}

Our proof of \cref{th:main} will use the following lemma.

\begin{lemma}\label{lemma:metric_empty}
  Let $F \in \cD(M)$. 
  We assume that a Hamiltonian function $H \colon T^*M \times I \to \bR$ satisfies $\supp(H_s) \cap \mathrm{RS}(F) = \emptyset$ for all $s\in I$. Then $F \simeq \Phi^H_1(F)$.
\end{lemma}
\begin{proof}
  We recall how to construct $K^H$.
  We first lift $H$ to a homogeneous Hamiltonian function $\widetilde H \colon (T^*(M\times\bR) \setminus 0_{M\times\bR}) \times I \to \bR$ by setting $\widetilde{H}_s(x,t;\xi,\tau) \coloneqq \tau H_s(x;\xi/\tau)$ for $\tau \neq 0$ and $\widetilde{H}_s=0$ when $\tau=0$. 
  Then we apply the results for homogeneous Hamiltonian isotopies in \cite{G-K-S}.  
  Composing $K^H$ with $F$ yields a sheaf $G$ on $M \times\bR\times I$ whose micro-support, outside the zero section, is given by 
  \begin{gather*}\MS(G) = \left\{
  (x,t,s;\xi,\tau,\sigma) \; \middle| \;  \begin{aligned}
  & \exists (x',t';\xi',\tau')\in \MS(F),\tau' \neq 0, \\ & (x,t;\xi,\tau) = \phi^{\widetilde H}_s(x',t';\xi',\tau'),\\ & \sigma = -\widetilde H_s(x,t;\xi,\tau) = -\tau H_s(x;\xi/\tau)\end{aligned} \right \}.
  \end{gather*}
  Since $\supp(H_s) \cap \RS(F) = \emptyset$ for all $s$, we see that the fiber variable $\sigma$ vanishes on
  $\MS(G)$.  
  By \cite[Prop.~5.4.5]{KS90} this implies that $G$ is the pull-back of a sheaf on $M\times\bR$.
  In particular $G|_{M\times\bR\times \{0\}} \simeq G|_{M\times\bR\times \{1\}}$, which is the claimed result.
\end{proof}

We now turn to the proof of \cref{th:main}. 

\begin{proof}[Proof of \cref{th:main}] 
    We first prove the inclusion $\gammasupp(L) \subset \mathrm{RS}(F_L)$. 
    Let $U$ be an open subset such that $U \cap \mathrm{RS}(F_L)=\emptyset$.
	For any $H$ such that $\supp(H_s) \subset U$ for any $s \in I$, by \cref{lemma:metric_empty} we get $d_{\cD(M)}(F_L,\Phi^H_1(F_L))=0$. By \eqref{eq:ineq-c}, we deduce 
	\[
		\gamma(L,\phi^H_1(L)) \le 2 d_{\cD(M)}(F_{L},\Phi^H_1(F_L))=0,
	\]
	hence $\phi^H_1(L)=L$. This proves that $U \cap \gammasupp(L)=\emptyset$ for any such open subset $U$. As a consequence $\gammasupp(L) \subset \mathrm{RS}(F_L)$.
	
	We next prove $\mathrm{RS}(F_L) \subset \gammasupp(L)$.
	As a first step, we establish the following.
	
	\begin{lemma}\label{lemma:boundary_inclusion}
		For $L \in \widehat{\frakL}(T^*M)$, one has
		\begin{equation}\label{eq:boundary_inclusion}
			\partial\, \mathrm{RS}(F_L) \subset \gammasupp(L),
		\end{equation}
		where $\partial$ means topological boundary, i.e., $\partial \mathrm{RS}(F_L)=\RS(F_L) \cap \Int(\mathrm{RS}(F_L))^c$.
	\end{lemma}
     
    \begin{proof}
        Let $U$ be an open subset such that $U \cap \gammasupp(L)=\emptyset$.  
        Then for any $H$ such that $\supp(H_s) \subset U$ for any $s \in I$, we have $L = \phi^H_1(L)$.  
        As recalled after~\eqref{eq:defQhat} we can lift $L$ to $L' \in \widehat{\cL}(T^*M)$ and we have $\phi^H_1(L') = T_c(L')$ for some $c$.
        By~\eqref{eq:QandPhi2} we deduce $T_c(F_L) \simeq \Phi^H_1(F_L)$, hence $\mathrm{RS}(F_L)=\phi^H_1(\mathrm{RS}(F_L))$.  
        Thus, either $U \cap \mathrm{RS}(F_L)=\emptyset$ or $U \subset \Int(\mathrm{RS}(F))$, which shows~\eqref{eq:boundary_inclusion}.
    \end{proof}
	
    We can now conclude the proof of \cref{th:main}. 
    To prove $\RS(F_L) \subset \gammasupp(L)$, it is enough to show that $\RS(F_L) \times 0_{\bS^1} \subset \gammasupp(L) \times 0_{\bS^1}$.
	Now we consider $L \times 0_{\bS^1} \in \widehat{\cL}(T^*(M \times {\bS^1}))$.
	Then we get $F_{L \times 0_{\bS^1}} \simeq F_L \boxtimes \bfk_{{\bS^1}}$, and hence $\RS(F_{L \times 0_{\bS^1}})=\RS(F_L) \times 0_{{\bS^1}}$, whose interior is empty.
	By \cref{lemma:boundary_inclusion}, we get 
	\[
		\RS(F_L) \times 0_{{\bS^1}}=\RS(F_{L \times 0_{\bS^1}}) \subset \gammasupp(L \times 0_{\bS^1}) \subset \gammasupp(L) \times 0_{\bS^1},
	\]
	where the last inclusion follows from \eqref{eq:gammasupp-product}. 
\end{proof}

\begin{remarks}\label{Rmks:3.3}\leavevmode
\begin{enumerate}
    \item \label{Rmk:3.3-1}
    Let $X, Y$ be two sets in a symplectic manifold. Define the following variant of Usher's distance (see \cite{usher2015observations}) between $X$ and $Y$ as 
    \[
        d_\gamma(X,Y)=\inf \{ \gamma(\varphi) \mid \varphi(X)=Y\}.
    \]
    Then we have for any $L, L'\in \widehat{\mathfrak L}(T^*M)$
    \[
        d_\gamma (\gammasupp(L), \gammasupp(L')) \leq \widehat\gamma (L,L'),
    \]
    where $\widehat \gamma(L,L')=\inf\{\gamma(\varphi)\mid \varphi(L)=L'\}$.
    
    In general $\gamma (L,L') \leq \widehat \gamma (L,L') $, but if we had an equality, this would mean (as in the proof of the theorem)
    \[
        d_\gamma (\mathrm{RS}(F_L), \mathrm{RS}(F_{L'})) \leq 2 d_{\cD(M)}  (F_L,F_{L'}).
    \]
    Then, we may conjecture that in general for all $F,G \in D(M)$ we have 
      \[
        d_\gamma (\mathrm{RS}(F), \mathrm{RS}(G)) \leq 2 d_{\cD(M)}  (F,G).
    \]
    \item \label{Rmk:3.3-2} In \cite{GV2022singular}, it is proved that the micro-support (hence also the reduced micro-support by Prop.~9.4 in \cite{viterbo2022supports}) is $\gamma$-coisotropic. 
    \item \label{Rmk:3.3-3} In \cite{Arnaud-Humiliere-Viterbo}, it is proved that there are indecomposable sets (i.e., compact connected sets that cannot be written as the union of two nontrivial compact connected) that appear as $\gamma$-supports. They can thus also appear as singular support of sheaves, and we may even impose that these are limits of constructible sheaves. 
\end{enumerate}
\end{remarks}

\section{An application to subdifferentials}\label{sec:subdifferential}

Let $f$ be a continuous function on $M$. 
In \cite{Vichery}, Vichery defined a subdifferential of $f$ at $x$ as follows.

\begin{definition}(\cite[Def.~3.4]{Vichery})
    The epigraph of $f$ is the set $Z_f=\{(x,t)\in M \times \mathbb R \mid f(x)\leq t\}$.
    Then $\partial f $ is defined as $-\RS (\bfk_{Z_f})$ and $\partial f(x)=-\RS (\bfk_{Z_f})\cap T_x^*M$ where $-A=\{(x,-p)\mid (x,p)\in A\}$.
\end{definition}
 
A more elementary definition from the same paper by Vichery is the following (\cite[Def.~4.6]{Vichery}).
 
\begin{proposition}
    The vector $\xi \in T_x^*M$ belongs to $\partial f(x)$ if and only $(x,\xi)$  belongs to the closure of the set of pairs $(y;\eta)$ such that setting $a=f(y)$ and $f_\eta(z)=f(z)-\langle \eta,z\rangle$ the map
    \[
        \lim_{\genfrac{}{}{0pt}{}{U\ni x}{ \varepsilon \to 0}}H^*(U\cap f_\eta^{<a+\varepsilon}) \longrightarrow \lim_{\genfrac{}{}{0pt}{}{U\ni x}{ \varepsilon \to 0}}H^*(U\cap f_\eta^{<a})
    \]
    is not an isomorphism.
\end{proposition}

We refer to \cite[Section~3.4]{Vichery} for the proof and the connection between this ``homological subdifferential'' and other subdifferentials, but notice that if $f$ is Lipschitz and  $\partial_Cf(x)$ is the Clarke differential at $x$ we have $\partial f(x) \subset \partial_C f(x)$ and the inclusion can be strict. 
 
Note that if $f$ is smooth, the graph of $df$ is an exact Lagrangian submanifold denoted by $\mathrm{graph}(df)$. 
Since $\gamma (\mathrm{graph}(df), \mathrm{graph}(dg)) = \max (f-g)-\min (f-g)= \mathrm{osc}(f-g)$, a $C^0$ Cauchy sequence of functions yields a Cauchy sequence in $\cL(T^*M)$, so that $\mathrm{graph}(df)$ is well defined in $\widehat{\cL}(T^*M)$ for any $f\in C^0(M, \mathbb R)$. 

\begin{proposition} 
    For any continuous function $f:M\to\bR$, we have 
    \[
        \gammasupp (\mathrm{graph}(df))= \partial f.
    \]
\end{proposition}

\begin{proof}
    By applying \cref{th:main} to $F=\bfk_{Z_f}$, we get $\RS(\bfk_{Z_f})=\gammasupp (\mathrm{graph}(df))$
    provided we prove that $\widehat Q (\mathrm{graph}(df))=\bfk_{Z_f}$. This of course holds if $f$ is smooth but needs to be established in the continuous case.
    
    We first claim that for any continuous functions $f,g$ we have:
    \[
        d_{\cD(M)}(\bfk_{Z_{g}},\bfk_{Z_{f}}) \leq 2 \Vert f-g\Vert_{C^0}.
    \]
    For two open sets $Z$ and $Z'$, there is a non-trivial morphism from $\bfk_Z$ to $\bfk_{Z'}$ if and only if $Z' \subset Z$.
    We set $\varepsilon \coloneqq \|f-g\|_{C^0}$, then we have $Z_f \subset Z_{g+\varepsilon}$ and $Z_g \subset Z_{f+\varepsilon}$.
    Since $T_c \bfk_{Z_f} \simeq \bfk_{Z_{f+c}}$ for $c \in \bR$, these inclusions imply that there exist canonical non-trivial morphisms $\bfk_{Z_f} \to T_\varepsilon \bfk_{Z_g}$ and $\bfk_{Z_g} \to T_\varepsilon \bfk_{Z_f}$, which give an $(\varepsilon,\varepsilon)$-isomorphism for the pair $(\bfk_{Z_f},\bfk_{Z_g})$.
    This proves the inequality. 

    Now let $f_n$ be a sequence of smooth functions $C^0$ converging to a continuous function $f$. 
    Then, by the above inequality $\bfk_{Z_{f_n}}$ converges to $\bfk_{Z_f}$ with respect to the distance $d_{\cD(M)}$ as $n$ goes to $+\infty$.
    This implies $\widehat Q (\mathrm{graph}(df))=\bfk_{Z_f}$ and concludes our proof.
\end{proof}

\begin{remark}
If $L \in \widehat{\mathfrak L}_{c}(T^*M)$, then $\gammasupp(L)\cap T_x^*M$ is non empty for all $x \in M$ (see \cite[Def.~6.4 and Prop.~6.10]{viterbo2022supports}). 
For $L=\mathrm{graph} (df)$, this means that $\partial f(x)$ is non-empty for all $x$. 
The condition $\mathrm{graph}(df) \in \widehat{\frakL}_{c}(T^*M)$ should correspond to $f$ being Lipschitz, in which case it is easy to see that $\partial f(x)\neq \emptyset$. 
\end{remark}

\printbibliography

\end{document}